\documentclass[a4paper]{amsproc} 
\usepackage{amssymb}
\usepackage{amscd}
\usepackage{pstricks,pst-plot,pst-node}
\newgray{grisclair}{0.87}
\newgray{grisfonce}{0.7}
\usepackage{epsfig}
\theoremstyle{plain}
 \newtheorem{thm}{Theorem}[section]
 
 \newtheorem{lem}{Lemma}[section]
 \newtheorem{cor}{Corollary}[section]
\theoremstyle{definition}

\theoremstyle{remark}
  
 \numberwithin{equation}{section}
\renewcommand{\leq}{\leqslant}
\renewcommand{\geq}{\geqslant}

\title[Riemann hypothesis]{A NEW REASON FOR DOUBTING THE RIEMANN HYPOTHESIS}



\author[Blanc]{\bfseries Philippe Blanc}

\address{
\hspace{-0.6 cm}D\'{e}partement des technologies industrielles \\ 
Haute \'{E}cole d'Ing\'{e}nierie et de Gestion  \\ 
CH-1400 Yverdon-les-Bains\\
Switzerland}
\email{philippe.blanc@heig-vd.ch}

\begin{document}
\vspace{18mm}
\setcounter{page}{1}
\thispagestyle{empty}

\begin{abstract}
We make plausible the existence of counterexamples to the Riemann hypothesis located in the neighbourhood of unusually large peaks of~$\vert \zeta \vert$. The main ingredient in our argument is an identity which links the zeros of a function $f$ defined on the interval $[-a,a]$ and the values of its derivatives of odd order at $\pm a$. 
\end{abstract}

\maketitle

\section{Introduction} 
\label{intro}
The Riemann hypothesis is the conjecture that all nontrivial zeros of the function~$\zeta$, defined for $\Re(s)>1$ by
\[
\zeta(s)=\sum_{n=1}^{\infty}\frac{1}{n^s}
\]
and on $\mathbb{C}\setminus \{1\}$ by analytic continuation, have real parts equal to $1/2$. The interested reader will find the history of the Riemann hypothesis in the book of Edwards \cite{Edwards} and, in the book of Borwein et al.\,\,\cite{Borwein}, the papers of Bombieri and Conrey in its favour and some reasons for doubting in an exciting paper of Ivi\'c.\\
\indent  The goal of this research, which began after having read Ivi\'c's paper, is to make plausible the existence of counterexamples to the Riemann hypothesis located in the neighbourhood  of unusually large peaks of $\vert \zeta \vert$, precisely where they are expected if they exist.\\
In this paper, we will make use of the Hardy $Z$ function defined by
\[
Z(t)=e^{i\theta(t)}\zeta \left(\frac{1}{2}+it\right)
\]
where
\[
\theta(t)=\arg\left( \pi^{-i\frac{t}{2}}\;\Gamma\left (\frac{1}{4}+i\frac{t}{2}\right)\right)
\]
and the argument is defined by continuous variation of $t$ starting with the value $0$ at $t=0$. It can be shown \cite{Ivic} that
\[
\theta(t) = \frac{t}{2}\log \frac{t}{2\pi} - \frac{t}{2} - \frac{\pi}{8}+
O\left(\frac{1}{t}\right).
\] The real zeros of $Z$ coincide with the zeros of $\zeta $ located on the line of real part $1/2$. If the Riemann hypothesis is true, then the number of zeros of $Z$ in the interval $ ( 0, t] $ is given by \cite{Ivic}
\begin{equation} \label{eq:N}
N(t)=\frac{1}{\pi}\theta(t)+1+S(t)
\end{equation}
where $S(t)=\frac{1}{\pi}\arg \zeta(1/2+it)$ if $t$ is not a zero of $Z$ and $\arg \zeta(1/2+it)$ is defined by continuous variation along the straight lines joining $2$, $2+it$ and $1/2+it$ starting with the initial value $\arg \zeta(2)=0$. If $t$ is a zero of $Z$ we set $S(t)=\lim_{\epsilon \to 0_+}S(t+\epsilon)$. Further, we introduce the function $S_1(t)=\int_0^tS(u)du$.\\
\indent The notations used in this paper are standard : $\lfloor x\rfloor$ and $\lceil x \rceil$ stand for the usual floor and ceiling functions and 
 $\{x\}:=x-\lfloor x\rfloor$. The writing $f(x)\ll g(x)$ means there exist $x_0$ and a positive constant $C$ such that $\left |f(x)\right |\leq Cg(x)$ for $x\geq x_0$ where $g(x)$ is positive for  $x\geq x_0$, $f(x)\gg g(x)$ is the same as $g(x)\ll f(x)$ and $f(x)\asymp g(x)$ is equivalent to $g(x)\ll f(x)\ll g(x)$. The symbol $f(x)=\Omega_{\pm}(g(x))$ means that $\limsup_{x\to \infty}f(x)/g(x)>0$ and $\liminf_{x\to \infty}f(x)/g(x)<0$.
 Bernoulli polynomials of degree $n$ are defined by
 \[
\int\limits_{x}^{x+1}B_{n}(t)\,dt=x^n\!.
 \]
Finally, for $r\geq 1$ and $t\geq t_r$ we set $log_1 t\equiv \log t$ and $\log_r t=\log(\log_{r-1}t)$ for $r\geq 2$.\\
\indent The content of this paper is as follows : In Section 2 we present our argument which disfavours the Riemann hypothesis. Statements and proofs of some technical results are to be found in Section 3.
\section{On the possible existence of counterexamples to the Riemann hypothesis}
Let $T_M>0$ be such that $Z$ reaches a maximum at $T_M$ and let $T\pm a$, located as indicated in the Figure 1, be such that $Z(T\pm a )\neq 0$ and $Z'$ reaches a local maximum at both $T+a$ and $T-a$. Further, let $\gamma_k$ be the zeros of~$Z$ numbered in such a way that $\gamma_{-2}< T-a<\gamma_{0}\leqslant\ldots\leqslant\gamma_n<T+a$ and $n$ is odd.\newpage
\vspace{8cm}
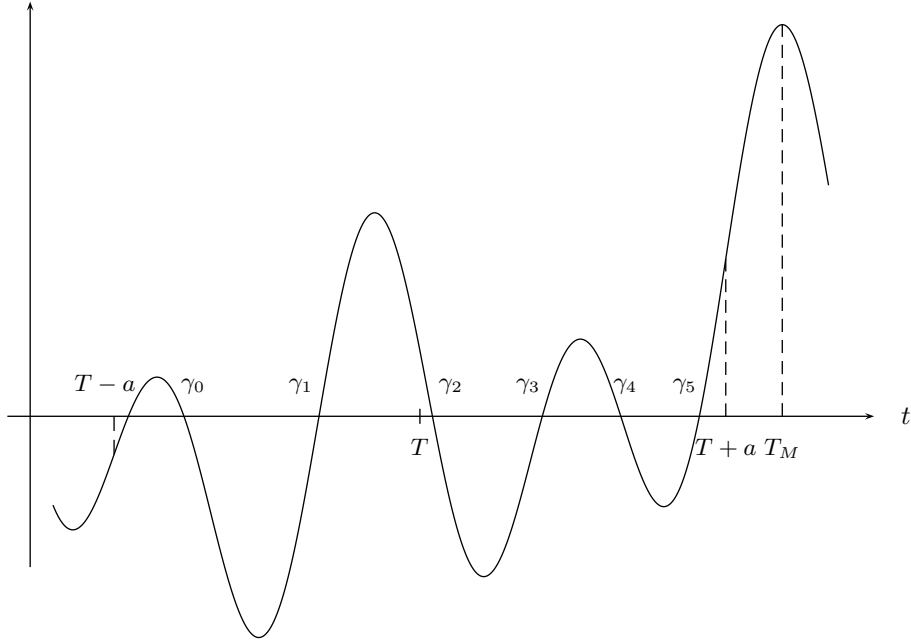
\begin{figure}[h!]\label{fig:situation}
\begin{center}
\begin{pspicture}(10.5,7)
 \psset{linewidth=0.5pt,xunit=0.6 cm,yunit=1 cm}
 \readdata{\mydataexp}{sit.dat}
    \psline{->}(-1,0)(18,0)
    \psline{->}(-0.5,-2)(-0.5,5.5)
    \dataplot[plotstyle=curve]
      {\mydataexp}
       \psline[linestyle=dashed](1.339,0)(1.339,-0.507)
        \uput{0.3}[90](1.1393,0){\small $T-a$}
      \psline[linestyle=dashed](14.747,0)(14.747,2.087)
        \uput{0.3}[270](14.747,0){\small $T+a$}
        \uput{0.3}[270]( 8.043,0){\small $T$}
        \psline(8.043,-0.1)(8.043,0.1)
         \uput{0.3}[90](3.07,0){\small $\gamma_{0}$}
           \uput{0.3}[90](5.431,0){\small $\gamma_{1}$}
           \uput{0.3}[90](8.718,0){\small $\gamma_{2}$}
             \uput{0.3}[90](10.415,0){\small $\gamma_{3}$}
         \uput{0.3}[90](12.546,0){\small $\gamma_4$}
         \uput{0.3}[90](13.846,0){\small $\gamma_5$}
       \psline[linestyle=dashed](15.986,0)(15.986,5.193)
         \uput{0.3}[270](15.986,0){\small $T_M$}
    \uput[0](18.3,0){$t$}
  \end{pspicture}
  \vspace{3cm}
    \caption{Positions of $T_M$, $T$ and $T\pm a$ with $n=5$.}
  \end{center}
  \end{figure}
  \noindent
As a consequence of Theorem \ref{Main} and Lemma \ref{Psinorm} there exist numbers $\alpha_{2k-1}^{\pm}$ depending on $T$, $a$ and $\gamma_{0},\ldots,\gamma_n$ such that for $K\geqslant n+1$ we have
\begin{equation}\label{eq:main}
\left|\;\sum_{k=1}^{K}\alpha_{2k-1}^+Z^{(2k-1)}(T+a)+\sum_{k=1}^{K}\alpha_{2k-1}^-Z^{(2k-1)}(T-a)\;\right|
\end{equation}
\[
\leqslant 2^{n-\frac{1}{2}}c_{2K-1,n}\left(\frac{2a}{n\pi}\right)^{2K}\max_{T-a\leqslant t \leqslant T+a}\vert Z^{(2K)}(t)\vert
\]
where $c_{2K-1,n}$ is defined in (\ref{eq:constc}).
Now, we assume that $T$ is sufficiently large and $a$ is small enough so that there exists  $K\geqslant n+1$ where   $K\!\in\left.\left(\frac{1}{4}\theta'(T),\frac{7}{8}\theta'(T)\right]\right.$ and  we define the numbers
\[
\beta_{2k-1}^{\pm}= (-1)^{k+1}\alpha_{2k-1}^{\pm}\theta'(T)^{2k-1}\,,\hspace{1cm}d_{2k-1}^{\pm}=(-1)^{k+1}\frac{Z^{(2k-1)}(T\pm a)}{\theta'(T)^{2k-1}}
\]
and
\[
e_{2K,n}=2^{n-\frac{1}{2}}c_{2K-1,n}\left(\frac{2a\theta'(T)}{n\pi}\right)^{2K}\min\left(\log T,\,3\, \zeta\left(\frac{1}{2}+\frac{2K}{\theta'(T)}\right)\right).
\]
Observe that the terms $\theta'(T)^{2k-1}$ and $1/\theta'(T)^{2k-1}$ are just scaling factors. 
Thanks to Lemma \ref{Psiprop} and since $n$ is odd, the numbers $\beta_{2k-1}^{\pm}$ are positive and  inequality (\ref{eq:main}) together with Corollary \ref{boundZ2k} imply that
\begin{equation}\label{eq:main2}
\left|\;\sum_{k=1}^{K}\beta_{2k-1}^+d_{2k-1}^{+}+\sum_{k=1}^{K}\beta_{2k-1}^-d_{2k-1}^{-}\;\right|\leqslant e_{2K,n}.
\end{equation}
\indent As examples, we compute the numbers $\beta_{2k-1}^{\pm}$, $d_{2k-1}^{\pm}$ and $e_{2K,n}$ for a couple of large values of $T$ for which $Z$ is computable and then we will see what is going on for very large values of $T$ if the Riemann hypothesis is true.\\ In these examples, we use Theorem \ref{boundZk} to compute the derivatives and, in the second one, we use the data of Bober and Hiary  \cite{Hiary}   together with finite differences to compute low order derivatives. In both cases, we neglect the error term.\\
As a first example, we choose 
\begin{equation}\label{T1}
T= 7.1934200352263711248\times 10^{14},
\end{equation}
$K=\left\lceil\frac{1}{4}\theta'(T)\right\rceil=5$ and $a=0.38644$ so that $n=3$ and $T_M=T+0.56179$, a value for which $Z(T_M)=738.756$. We get $e_{10,3}=2940.12$ and  the values of  $\beta_{2k-1}^{\pm}$ and $d_{2k-1}^{\pm}$ are collected in Table 1.
\begin{table}[htb]\label{tab:D} 
\centering\begin{tabular}{|c|r|r|r|r|r}
\hline
 $2k-1$ & $\beta_{2k-1}^+ $ & $d_{2k-1}^+ $ & $\beta_{2k-1}^- $ & $d_{2k-1}^- $\\
 \hline
1     & 0.20  & 171.73  & 1.13  &0.32    \\
3     & 0.85  &  45.20  & 2.60  &0.60    \\
5     & 2.25  &  17.90  & 4.46  &1.06    \\
7     & 4.86  &   8.99  & 7.31  &1.55    \\
9     & 9.53  &   5.26  & 12.08 &1.94    \\
\hline
\end{tabular}
\vspace{0.5cm}
\caption{Table of $\beta_{2k-1}^{\pm}$ and $d_{2k-1}^{\pm}$ for $T$ given by (\ref{T1}).}
\end{table}

\noindent For the second one, we choose 
\begin{equation}\label{T2}
T= 3.924676458989430915525116928410362023\times 10^{31},
\end{equation}
$K=\left\lceil\frac{1}{4}\theta'(T)\right\rceil=9$ and $a=0.33794$ so that $n=7$ and $T_M=T+0.43039$, a value for which $Z(T_M)=16244.865$, the largest known value of $Z$.  We get $e_{18,7}=100375.07$ and  the values of  $\beta_{2k-1}^{\pm}$ and $d_{2k-1}^{\pm}$ are collected in Table~\ref{tab:B}.\\ 
\begin{table}[htb]\label{tab:B}
\centering\begin{tabular}{|r|r|r|r|r|r}
\hline
 $2k-1$ & $\beta_{2k-1}^+ $ & $d_{2k-1}^+ $ & $\beta_{2k-1}^- $ & $d_{2k-1}^- $\\
 \hline
1    &   0.10  & 3209.26   &     8.79  &   0.24    \\
3    &   0.48  &  622.68   &    16.76  &   0.38    \\
5    &   1.30  &  180.96   &    21.37  &   0.53    \\
7    &   2.63  &   66.25   &    23.76  &   0.71    \\
9    &   4.52  &   28.01   &    25.25  &   0.93    \\
11   &   6.99  &   12.92   &    26.63  &   1.19    \\
13   &  10.05  &    6.21   &    28.33  &   1.48    \\
15   &  13.74  &    2.94   &    30.60  &   1.78    \\
17   &  18.12  &    1.23   &    33.57  &   2.07    \\
\hline
\end{tabular}
\vspace{0.2cm}
\caption{Table of $\beta_{2k-1}^{\pm}$ and $d_{2k-1}^{\pm}$ for $T$ given by (\ref{T2}).}
\end{table}
\indent
\hspace{-0.3 cm}We now turn to the key of our argument. We observe, in Tables 1 and 2, that $\beta_1^+d_1^+$ is much smaller than $e_{10,3}$ and $e_{18,7}$. Surprisingly, under the Riemann hypothesis and some technical conditions, if $S(T+a)-S(T-a)$ is large then not only $\beta_1^+$ is large, but numerical computations strongly suggest that $\beta_1^+d_1^+$ can be much larger than~$e_{2K,n}$.\\ 
\indent
From now we assume that the Riemann hypothesis is true and for the moment we relax the condition that $Z'$ be maximum at $T\pm a$. Concerning the order of $S(t)$, it is known \cite{Carneiro}, \cite{Montgomery}, that 
\[
\vert S(t)\vert\leqslant \left( \frac{1}{4}+o(1)\right)\frac{\log t}{\log_2 t}\,,\hspace{1cm}  S(t)=\Omega_{\pm}\left( \sqrt{\frac{\log t}{\log_2 t}}\,\;\right)
\]
and a conjecture of Farmer, Gonek and Hughes \cite{Farmer} asserts that
\begin{equation}\label{Farmer}
\limsup_{t\to \infty}\frac{S(t)}{\sqrt{\log t\,\log_2 t}}=\frac{1}{\pi \sqrt{2}}\,.
\end{equation}
For a further use we introduce the functions
\[
A_1(t)= \frac{1}{4}\frac{\log t}{\log_2 t}\,,\hspace{1cm}A_2(t)= \frac{\sqrt{\log t\,\log_2 t}}{\pi \sqrt{2}}\,,\hspace{1cm}A_3(t)=\sqrt{\frac{\log t}{\log_2 t}}\,.
\]
Let $M_+(t)=\max_{0\leqslant u\leqslant t}S(u)$ and let $c\!\in (0,1)$. Under the assumptions of Lemma~\ref{DeltaS}, there exist arbitrary large $T$ and $a\asymp (\log_2 T)^{-1}\log_3 T$ such that $S(T+a)\geqslant c\,M_+(T+a)$ and
\[
\beta_1^+\geqslant e^{\textstyle{
 cM_+(T+a)\left(1+O\left(\frac{1}{\log_3T}\right)\right)}}
\]
\\
and moreover, if $M_+(T+a)=A_1(T+a)$ and  $K\!\in\left.\left(\frac{1}{4}\theta'(T),\frac{7}{8}\theta'(T)\right]\right.$, then
\[
 e_{2K,n} \leqslant e^{\textstyle{ -\frac{2\pi c K}{\log_3 T}\left(1+O\left(\frac{1}{\log_3T}\right)\right)}}.
\]
This shows that when $M(T+a)=A_1(T+a)$ then $\beta_1^+$ is large and $ e_{2K,n}$ is small when $T$ is large. Our lower bound for $\beta_1^+$ does not really take into account the distribution of zeros, and proving that $\beta_1^+d_1^+$ can be much larger that $e_{2K,n}$ if $M_+(T+a)=A_2(T+a)$ or $A_3(T+a)$ requires to solve the minimization problem~(P) defined in Section~3. Its solution leads to 
\begin{equation}\label{phic}
\beta_1^+\geqslant \phi_c(T,a,M_+(T+a),r_1(T,a),r_2(T,a))
\end{equation}
where $r_1(T,a)$ and $r_2(T,a)$ are $O((\log_3 T)^{-1})$ and the function $\phi_c$ has to be computed numerically.\\
We now seek a lower bound for $d_1^+$.
Concerning the order of $\zeta(1/2+it)$, a conjecture~\cite{Farmer}, related to (\ref{Farmer}), asserts that there exist arbitrary large values of $t$ such that
\begin{equation}\label{FarmerB}
\vert \zeta(\frac{1}{2}+it)\vert  \geqslant e^{\textstyle{(1+o(1))\sqrt{\frac{1}{2}\log t\log_2 t}}}
\end{equation}
whereas Bondarenko and Seip \cite{Bondarenko} show unconditionally that
\begin{equation}\label{Bondarenko}
\vert \zeta(\frac{1}{2}+it)\vert \geqslant e^{\textstyle{(1+o(1))\sqrt{\frac{\log t\,\log_3 t}{2\log_2t}}}}
\end{equation} 
for some arbitrary large $t$.\\
We quote Bober and Hiary \cite{Bober}:
"It is always the case in our computations that when $\zeta(1/2+it)$ is very large there is  a  large  gap  between  the  zeros  around  the  large  value. And  it  seems  that  to
compensate for this large gap the zeros nearby get "pushed" to the left and right.
A typical trend in the large values that we have found is that
$S(t)$ is particularly
large and positive before the large value and large and negative afterwards."\\
Therefore, it is reasonable to expect that for $c \in (0,1)$ small enough, there exist arbitrary large  $T$, $a$ and $T_M$ as defined in the beginning of this section such that (\ref{phic}) holds and $Z(T_M)$ is large. Further, assuming that $S(T_M)\geqslant -M_+(T+a)$ we get, using (\ref{eq:N})
\[
d_1^+=\frac{Z'(T+a)}{\theta'(T)}\geqslant \frac{Z(T_M)}{\theta'(T)(T_M-\gamma_n)}\geqslant \frac{Z(T_M)}{2\pi M_+(T+a)}.
\]
Now it gets interesting. We compute our different bounds for $c=1/4$, $T=10^{20000}$,  
$a=(\log_2 T)^{-1}\log_3 T=0.22107$ and $K=\lfloor\frac{7}{8}\theta'(T)\rfloor =20146$ without, of course, taking into account the $O(\cdot)$ and $o(\cdot)$.\\
If $M_+(T+a)=A_2(T+a)$, it is reasonable to use the bound (\ref{FarmerB}) and since $\vert Z(t)\vert =\vert \zeta(1/2+it)\vert$, this leads to
\[
\beta_1^+\geqslant 8.55\times 10^{864},\hspace{0.5 cm}
d_1^+\geqslant 8.86\times 10^{212}\hspace{0.5 cm}\mbox{and}\hspace{0.5 cm}e_{2K,n}= 5.18\times 10^{775}
\]
so that $\beta_1^+d_1^+\geqslant 7.57\times 10^{1077}$.\\
If $M_+(T+a)=A_3(T+a)$, we use (\ref{Bondarenko}) to get
\[
\beta_1^+\geqslant 2.88\times 10^{874},\hspace{0.5 cm}
d_1^+\geqslant 2.34\times 10^{28}\hspace{0.5 cm}\mbox{and}\hspace{0.5 cm}e_{2K,n} = 9.97\times 10^{896}
\]
so that $\beta_1^+d_1^+\geqslant 6.73\times 10^{902}$.\\
As a consequence of Lemma \ref{opt}, the distribution of zeros of $Z$ which gives rise to the lower bound (\ref{phic}) is not the one mentioned in the previous citation. This suggests that $\beta_1^+$ is probably much larger than the right hand side of (\ref{phic}).\\
For very large $T$, the numbers $d_{2k-1}^{\pm}$ are of course not all positive but when $d_{2k-1}^{\pm}$ is negative then $\vert d_{2k-1}^{\pm}\vert$ is very probably small with respect to $d_1^+$ as suggested by our second example in which $-1< d_{2k-1}^+ <0$ and $d_{2k-1}^->0$ for $k=11,12,\ldots,24$. In the general case, some of the $\beta_{2k-1}^{\pm}d_{2k-1}^{\pm}$ must be negative and large in absolute value in order to compensate almost exactly $\beta_1^+d_1^+$. This is hardly conceivable. Note that our argument fails if $S(T+a)$ is "small" each time that $Z(T_M)$ is very large, a situation which is not observed in \cite{Bober}. Anyway, 
our results cast doubt on the numerical evidence in favour of the Riemann hypothesis.  
\section{Some technical results}
The key of our argument is an identity which links the zeros of a function $f\in C^{2r}[-a,a]$ and the values of its derivatives of odd order on the boundaries of the interval. The proofs of Theorem \ref{Main} and Lemmas \ref{Psiprop} and \ref{Psinorm} are very similar to those of Lemmas 2.3 and 2.7 of \cite{Blanc2} and Lemma 3.2  of \cite{Blanc1}.
\begin{thm}\label{Main}
Let $\Psi_{2l-1}$ be defined for pairwise distinct $x_{0},\ldots,x_n\in ( -a,a)$ by 
\[
\Psi_{2l-1}(x_{0},\ldots,x_n,x)=\frac{(4a)^{2l-1}}{(2l)!}\,\sum_{k=0}^{n}\,\mu_k\,\Big{(}B_{2l}\big{(}\frac{1}{2}+\frac{x+x_k}{4a}\big{)}+B_{2l}\big{(}\Big{\{}\frac{x-x_k}{4a}\Big{\}}\big{)}\Big{)}
\]
where  
\[
\mu_k=\frac{1}{\rule[3 mm]{0 cm}{4 mm}\displaystyle \prod_{\genfrac{}{}{0 cm}{} {0\leq j\leq n}{\; j\neq k}}\Big{(}\sin\left(\pi \frac{x_k}{2a}\right)-\sin\left(\pi \frac{x_j}{2a}\right)\Big{)}}\hspace{0.5cm} \mbox{for }k=0,\ldots,n.
\]
Then
\begin{enumerate}
\item[a)] For $l\geq 1$ the functions $\Psi_{2l-1}(\cdot,\ldots,\cdot,\pm a)$ have continuous extensions $\Psi_{2l-1}^*(\cdot,\ldots,\cdot,\pm a)$ defined for $x_{0},\ldots,x_n\in (-a,a)$.
\item[b)] For $2l\geq n+2$ the functions $\Psi_{2l-1}$ have continuous extensions $\Psi_{2l-1}^*$ defined for $x_{0},\ldots,x_n\in ]\negmedspace -a,a[$ and $x\in [ -a,a]$.
\item[c)] If $2r\geq n+2$ and $f\in C^{2r}[-a,a]$ is a function which vanishes at $x_k$ where $-a<x_{0}\leqslant \ldots\leqslant x_n<a$ and the $x_k$ are numbered according to their multiplicities, then we have the identity 
\begin{eqnarray}\label{eq:mainbis}
\sum_{k=1}^r \Psi_{2k-1}^*(a)f^{(2k-1)}(a)&-&\sum_{k=1}^{r}\Psi_{2k-1}^*(-a)f^{(2k-1)}(-a)\\
&=&\int\limits_{-a}^{a}\Psi_{2r-1}^*(x)f^{(2r)}(x)\, dx\nonumber
\end{eqnarray}
where for short $\Psi_{2k-1}^*(\pm a)$ and $\Psi_{2r-1}^*(x)$ stand for $\Psi_{2k-1}^*(\cdot,\ldots,\cdot,\pm a)$ and $\Psi_{2r-1}^*(\cdot,\ldots,\cdot,x)$\,.
\end{enumerate}
\end{thm}
\vspace{0.3 cm}
\begin{lem}\label{Psiprop}
Let $\Psi_{2l-1}^*$ be the function defined in Theorem \ref{Main}.
Then
\[
(-1)^{l+1}\Psi_{2l-1}^*(a)>0\,\,\,\mbox{and}\,\,\,(-1)^{n+l+1}\Psi_{2l-1}^*(-a)>0\,.
\]
\end{lem}
\vspace{0.3 cm}
\begin{lem}\label{Psinorm}
Let $\Psi_{2l-1}^*$ be the function defined in Theorem \ref{Main} and assume that $l\geqslant n+1$. Then
\[
\vert\vert\Psi_{2l-1}^*\vert\vert_2 \leq \;\frac{2^{n-1}\;c_{2l-1,n}}{ \sqrt{a}}\left(\frac{2a}{n\pi}\right )^{2l}
\]
where 
\[
\vert\vert\Psi_{2l-1}^*\vert\vert_2^2 = \int\limits_{-a}^a\left (\Psi_{2l-1}^*(x)\right )^2\,dx
\]
and
\begin{equation}\label{eq:constc}
c_{2l-1,r}=\left(\sum_{s=0}^{\infty}\left (\left (\frac{r}{r+s}\right ) ^{2l-1}\binom{2r+s-1}{s}\right )^2\right)^{\frac{1}{2}}\,.
\end{equation}
Moreover, if $n$ is sufficiently large and $l\geqslant n\frac{\log n}{\log_2 n}$ then $c_{2l-1,n}=e^{o(n)}$ where the $o(\cdot)$ is uniform in $l$.
\end{lem}
The next theorem \cite{Blanc3} and its corollary, which hold without assuming the Riemann hypothesis, are concerned with the high order derivatives of the Hardy $Z$ function.
\begin{thm}\label{boundZk}
Let $t$ be large enough and $c\geq e$ be a fixed constant. Then, for $0\leq k\leq  \frac{c}{2}\,\theta'(t)$, we have, uniformly in $k$,
\begin{eqnarray*}
Z^{(k)}(t)&=&2\hspace{-2mm}\sum_{1\leq n \leq \sqrt{\frac{t}{2\pi}}}\frac{1}{\sqrt{n}}\left ( \theta'(t)-\log n\right)^{k}\cos\left(\theta(t)-t\log n+k\frac{\pi}{2}\right)
\nonumber \\& +& O\left(t^{ \frac{1}{4}(c\log c -c -1)}\theta'(t)^{k+1}\right).
\end{eqnarray*}
\end{thm}
\begin{cor}\label{boundZ2k}
Let T be large, $0<a\ll 1$ and $K\!\in\!\!\left.\left(\frac{1}{4}\theta'(T),\frac{7}{8}\theta'(T)\right]\right.$. Then, for $t\in [T-a,T+a]$, we have
\[
\vert Z^{(2K)}(t)\vert\leqslant \min\left(\log T,\,3\, \zeta\left(\frac{1}{2}+\frac{2K}{\theta'(T)}\right)\right)\theta'(T)^{2K}.
\]
\end{cor}
\begin{proof}
For $t\in[T-a,T+a]$ and since $2K\leqslant \frac{7}{4}\theta'(T)$ and $\theta''(t)=O\left(\frac{1}{t}\right)$, we can use Theorem \ref{boundZk} with $c=\frac{7}{2}+O\left(\frac{a}{T\log T}\right)$ to get the bound
\[
\vert Z^{(2K)}(t)\vert\leq \left( 2\left(\sum_{1\leq n \leq \sqrt{\frac{t}{2\pi}}}\frac{1}{\sqrt{n}}
\left( 1-\frac{\log n}{\theta'(t)}\right)^{2K}\right)+O\left(t^{-\frac{1}{50}}\theta'(t)\right)\right)\theta'(t)^{2K}\,.
\]
Moreover $2K> \frac{1}{2}\theta'(T)$ and $\theta'(t)=\theta'(T)(1+O(\frac{a}{T\log T}))$ and we hence have the inequalities
 \[
 \left(1-\frac{\log n}{\theta'(t)}\right)^{2K}\leq e^{ -2K\textstyle\frac{\log n}{\theta'(t)}}= n^{\textstyle-\frac{2K}{\theta'(T)}+O(\frac{a}{T\log T})}\leq n^{-\frac{2K}{\theta'(T)}}\left(1+O\left(\frac{a}{T}\right)\right)
 \]
for $1\leq n \leq~\sqrt{\frac{t}{2\pi}}$\,. We complete the proof using
\[
\sum_{1\leq n \leq \sqrt{\frac{t}{2\pi}}}\frac{1}{n}=\log \sqrt{\frac{t}{2\pi}}+\gamma +O\left(\frac{1}{\sqrt{t}}\right)=\frac{1}{2}\log T -\frac{1}{2}\log 2\pi +\gamma+O\left(\frac{1}{\sqrt{T}}\right)
\]
where $\gamma$ is Euler's constant and
\[
\sum_{1\leq n \leq \sqrt{\frac{t}{2\pi}}}\frac{1}{n^{\frac{1}{2}+\frac{2K}{\theta'(T)}}}\leqslant \zeta\left(\frac{1}{2}+\frac{2K}{\theta'(T)}\right)
\]
together with $\theta'(t)^{2K}=\theta'(T)^{2K}(1+O(\frac{a}{T}))$.
\end{proof}
\vspace{0.5cm}
\noindent Applied to the Hardy $Z$ function and the interval $[T-a,T+a]$, identity (\ref{eq:mainbis}) writes here
\begin{eqnarray*}\label{eq:mainZ}
\sum_{k=1}^{K}\Psi_{2k-1}^*(a)Z^{(2k-1)}(T+a)&-&\sum_{k=1}^{K}\Psi_{2k-1}^*(-a)Z^{(2k-1)}(T-a)\\
&=&\int\limits_{-a}^{a}\Psi_{2K-1}^*(x)Z^{(2K)}(T+x)\, dx\nonumber
\end{eqnarray*}
where $\Psi_{2k-1}^*(\pm a)$ and $\Psi_{2K-1}^*(x)$ stand respectively for $\Psi_{2k-1}^*(x_{0},\ldots,x_n,\pm a)$ 
and $\Psi_{2K-1}^*(x_{0},\ldots,x_n,x)$ and  $x_k = \gamma_k -T$.\\
We set $\alpha_{2k-1}^{\pm}=\pm \Psi_{2k-1}^*(\pm a)$ and we use the Cauchy-Schwarz inequality
together with Lemma \ref{Psinorm} to get the relation  (\ref{eq:main}). Lemma \ref{Psiprop} implies the positivity of $\beta_{2k-1}^{\pm}$ since $n$ is odd and (\ref{eq:main2}) is a consequence of Corollary \ref{boundZ2k} and relation (\ref{eq:main}).\\
Before turning to the computation of lower bounds for $\beta_1^+$, we first prove some preparatory results.
\begin{thm}\label{S1}
Let $M_0(t)=\displaystyle\max_{0\leq u \leq t}\vert S(u)\vert$ and assume that $M_0(2t)\hspace{-0.05 cm}=\hspace{-0.05 cm} O(M_0(t))$. Then
\[
S_1(t)=O\left(\frac{M_0(t)}{\log_2 t}\right).
\]
\end{thm}
\begin{proof}
The proof is very similar to a proof of a classical result of Titchmarsh \cite[Theorem 14.13]{Titchmarsh} so we only give a sketch. We will make use of the formula
\begin{equation}\label{eq:S1}
S_1(T)=\frac{1}{\pi}\int\limits_{\frac{1}{2}}^{2}\log\vert \zeta(\sigma+iT)\vert\,d\sigma +O(1).
\end{equation}
Let $M_1(t)=\max_{0\leq u\leq t}\vert S_1(u)\vert$. For $\sigma>\frac{1}{2}$ and $0<\xi<\frac{1}{2}t$, we have
\begin{equation}\label{eq:tit}
\log\zeta(\sigma +it)=i\int\limits_{t-\xi}^{t+\xi}\frac{S(y)}{\sigma-\frac{1}{2}+i(t-y)}\,dy+O\left(\frac{M_1 (2t)}{\xi}\right)+O(1)
\end{equation}
and on the set $\Omega=\{\sigma+it\in\mathbb{C}\vert \,\sigma\geqslant \frac{1}{2}+\frac{1}{\log_2T},\,4\leqslant t \leqslant T\}$ we get
\begin{eqnarray*}
\vert\log\zeta(\sigma +it)\vert&\leqslant& M_0(2 t )\int\limits_{t-\xi}^{t+\xi}\frac{dy}{((\sigma-\frac{1}{2})^2+(t-y)^2)^{\frac{1}{2}}}+O\left(\frac{M_1 (2t)}{\xi}\right)+O(1)\\
                     &\leqslant& 2 M_0( 2t )\frac{\xi}{\sigma-\frac{1}{2}}+O\left(\frac{M_1 (2t)}{\xi}\right)+O(1)\\
                     &\leqslant& 2M_0(2 T )(\log_2T)\xi +O\left(\frac{M_1 (2T)}{\xi}\right)+O(1)
\end{eqnarray*}
and the choice $\xi=\displaystyle \left(\frac{M_1(2T)}{M_0( 2T)\log_2T}\right)^{\frac{1}{2}}$ leads to
\[
\log\zeta(\sigma +it)=O\left( (M_0(T)M_1(2T) \log_2 T)^\frac{1}{2}\right)
\]
since we assume that $M_0(2t)=O(M_0(t))$.
Now we apply Hadamard's three-circles theorem to the circles $C_1$, $C_2$ and $C_3$ centred in $\sigma_1+it$ where $\sigma_1=\frac{3}{2}+\frac{1}{\log_2 T}$ of radii $r_1=\sigma_1-\frac{5}{4}$, $r_2=\sigma_1-\sigma$ and $r_3=1$ where~$\sigma < \frac{5}{4}$. Using the facts that $\log M_0(T)\gg\log_2 T$ and $\log M_1(T)\gg\log_2T$ we get
\[
\log\zeta(\sigma +it)=O\left(( M_0(T)M_1(2T+2) \log_2 T)^\frac{1}{2}e^{-C(\sigma-\frac{1}{2})\log_2 T}\right)
\]
for some positive constant $C$, which implies that
\begin{equation}\label{S1part}
\int\limits_{\frac{1}{2}+\frac{1}{\log_2 T}}^{2}\log\vert \zeta(\sigma +it)\vert \,d\sigma=O\left(( M_0(T)M_1(3T))^\frac{1}{2}(\log_2 T)^{-\frac{1}{2}}\right).
\end{equation}
The real part of (\ref{eq:tit}) may be written
\begin{eqnarray*}
\log\vert \zeta(\sigma +it)\vert &=& \int\limits_{0}^{\xi}\frac{x}{(\sigma-\frac{1}{2})^2+x^2}(S(t-x)-S(t+x))\,dx\\
&+&O\left(\frac{M_1(2t)}{\xi}\right)+O(1)\nonumber
\end{eqnarray*}
and we have
\begin{eqnarray*}
\int\limits_{\frac{1}{2}}^{\frac{1}{2}+\mu} \log\vert \zeta(\sigma +it)\vert\,d\sigma &=& \int\limits_{0}^{\xi}\mbox{Arctan}\left(\frac{\mu}{x}\right)(S(t-x)-S(t+x))\,dx\\
+O\left(\mu\frac{M_1(2t)}{\xi}\right)
&+&O(\mu)=O(\xi M_0(2T))+O\left(\mu\frac{M_1(2T)}{\xi}\right)+O(\mu).
\end{eqnarray*}
Finally, we choose $\mu=\frac{1}{\log_2 T}$ and $\xi$ as before to get
\[
\int\limits_{\frac{1}{2}}^{\frac{1}{2}+\frac{1}{\log_2 T}}\log\vert \zeta(\sigma +it)\vert \,d\sigma=O\left(( M_0(T)M_1(2T))^\frac{1}{2}(\log_2 T)^{-\frac{1}{2}}\right)
\]
which together with (\ref{S1part}) and (\ref{eq:S1}) yield
\[
S_1(t)=O\left(( M_0(T)M_1(3T))^\frac{1}{2}(\log_2 T)^{-\frac{1}{2}}\right).
\]
To complete the proof,  we introduce the function
\[
\Psi(T)=\max_{4\leqslant t \leqslant T}\frac{(\log_2 t)M_1(t)}{M_0(t)}
\]
and we proceed as Titchmarsh.
\end{proof}
\begin{lem}\label{meanxk}
Let $x_k = \gamma_k -T$ and $0<a\ll 1$. Then
\[
\sum_{k=0}^n \sin \left(\pi \frac{x_k}{2a} \right)= S(T+a)+S(T-a)+O \Big{(}\frac{\displaystyle M_1(T+a)}{a}\Big{)}
\]
where $M_1(t)=\max_{0\leq u\leq t} \vert  S_1(u)\vert$.
\end{lem}
\begin{proof}[Proof of Lemma \ref{meanxk}]
Using Stieltjes integral we have 
\begin{eqnarray*}
\sum_{k=0}^n \sin \left(\pi \frac{x_k}{2a} \right)&=&  \sum_{T-a< \gamma_k <T+a} \sin \left(\pi \frac{\gamma_k-T}{2a}\right) -s_{-1}\\ 
& =&\int\limits_{T-a}^{T+a}\sin \left(\pi \frac{t-T}{2a} \right)\; d\left (\frac{1}{\pi}\theta(t)+1+S(t) \right)-s_{-1}
\end{eqnarray*}
where $s_{-1}=0$ or $\sin \left(\pi \frac{\gamma_{-1}-T}{2a} \right) $ according to the position of $\gamma_{-1}$ with respect to $T-a$, and integrations by parts lead to
\begin{eqnarray*}
\sum_{k=0}^n \sin \left(\pi \frac{x_k}{2a} \right) &= &\frac{2a}{\pi^2}\int\limits_{T-a}^{T+a}\cos \left(\pi \frac{t-T}{2a}\right)\theta''(t)\,dt\;+\;S(T+a)+S(T-a)\\
&-&\frac{\pi}{2a}\int\limits_{T-a}^{T+a}\cos \left(\pi \frac{t-T}{2a}\right)S(t)\,dt\,-s_{-1}.
\end{eqnarray*}
We complete the proof using the estimates $\theta''(t)=0 \Big{(}\frac{1}{t}\Big{)}$ and the second mean value theorem.
\end{proof}

\begin{lem}\label{DeltaS}
Let $M_+(t)=\max_{0\leqslant u\leqslant t}S(u)$ and $M_-(t)=\min_{0\leqslant u\leqslant t}S(u)$ and assume that $M_-(t)=O(M_+(t))$ and $M_+(2t)=O(M_+(t))$. Then for $t$ sufficiently large there exists $\alpha\in[1,2]$, depending on $t$, such that 
\[
 S(t-\alpha H)\ll \frac{M_+(t)}{\log_3 t}
\]
where $H= c_H\frac{\log_3 t}{\log_2 t}$ and $c_H$ is a fixed positive constant.
\end{lem}
\begin{proof}
we have
\[
\left(\sup_{[t-2H,t-\frac{3}{2}H]}S(u)\right)H\geq \int\limits_{t-2H}^{t-\frac{3}{2}H}S(u)\,du\geqslant -C\frac{M_+(t)}{\log_2 t}
\]
and
\[
\left(\inf_{[t-\frac{3}{2}H,t-H]}S(u)\right)H\leq \int\limits_{t-\frac{3}{2}H}^{t-H}S(u)\,du\leqslant C \frac{M_+  (t)}{\log_2 t}
\]
for some positive constant $C$ and hence there exists $\alpha \in [1,2]$ such that
\[
 S(t-\alpha H)\ll \frac{M_+(t)}{\log_3 t}\,.
 \]
\end{proof}
\begin{thm}\label{Lowerbeta1}
Assume that $M_+(t)$ and $M_-(t)$ satisfy the assumptions of Lem-ma~\ref{DeltaS} and let $c\in(0,1)$ be a fixed constant. Then there exist arbitrary large $T$ and $a\asymp (\log_2 T)^{-1}\log_3 T$ such that
\begin{equation}\label{Boundbeta1bis}
\beta_1^+\geqslant e^{\textstyle{
 cM_+(T+a)\left(1+O\left(\frac{1}{\log_3T}\right)\right)}}\,.
\end{equation}
Moreover, if $M_+(T+a)=A_1(T+a)$ and $K\!\in\left.\left(\frac{1}{4}\theta'(T),\frac{7}{8}\theta'(T)\right]\right.$, then
\begin{equation}\label{Bounde2K}
 e_{2K,n} \leqslant e^{\textstyle{ -\frac{2\pi c K}{\log_3 T}\left(1+O\left(\frac{1}{\log_3T}\right)\right)}}.
\end{equation}
\end{thm}
\begin{proof}
Let $t$ be sufficiently large, such that $S(t)=cM_+(t)$, and $\alpha$ satisfying the conclusion of Lemma~\ref{DeltaS} for $H=\frac{1}{4}(\log_2 t)^{-1}\log_3 t$. Further let $T$ and $a$ such that $T+a=t$ and $T-a=t-\alpha H$ and hence $a=c_1(\log_2 t)^{-1}\log_3 t$ for some $c_1\in[\frac{1}{8},\frac{1}{4}]$. This choice of $a$ will be used in the proof of (\ref{Bounde2K}). By definition
$
\beta_1^+=\alpha_1^+\theta'(T)=\Psi_1^*(x_{0},\ldots,x_n,a)\theta'(T)
$
where $x_k = \gamma_k -T$
 and proceeding as in the proof of \cite[Lemma 3.2]{Blanc1}, one checks that for pairwise distinct $y_k\!\in(-a,a)$ we have
\[
\Psi_1(y_{0},\ldots,y_n,a)\geqslant\frac{2a(1+\bar{t})^n}{3 \pi^2n^2}
\]
where
\[
\bar{t}=\frac{1}{n+1}\sum_{k=0}^n\sin\left(\pi \frac{y_k}{2a}\right)
\]
which, by continuity, extends to $\Psi_1^*(y_{0},\ldots,y_n,a)$ when the $y_k$ are not distinct.
From (\ref{eq:N}) we deduce that
\begin{equation}\label{mplusn}
n=\frac{2a}{\pi}\theta'(T)\left(1+O\left(\frac{1}{\log_3 T}\right)\right)
\end{equation}
and hence Theorem \ref{S1} and Lemma \ref{meanxk} imply that $\bar{t}=O((\log_3T)^{-1})$ and
\begin{eqnarray*}
\log \beta_1^+&>&n\log(1+\bar{t})+O(\log_2 T)= n\,\bar{t}\,(1+O(\bar{t}))+O(\log_2 T)\\
&=& \left(S(T+a)+S(T-a)+O\left(\frac{M_+(T+a)}{\log_3T}\right)\right)\left(1+O\left(\frac{1}{\log_3T}\right) \right)\\
&=&cM_+(T+a)\left(1+O\left(\frac{1}{\log_3T}\right)\right).
\end{eqnarray*}
By the choice of $a$ and (\ref{mplusn}) we have
\[
K>n\frac{\log n}{\log_2 n}
\] 
and, by Lemma \ref{Psinorm}, we get $\log c_{2K-1,n}=o(n)$ and using
\[
 n\pi=2a\theta'(T)+ \pi cA_1(T+a)\left(1+O\left(\frac{1}{\log_3T}\right)\right) 
\]
we conclude that
\begin{eqnarray*}
\log e_{2K,n}&=& n\log 2 +o(n) -2K\log\frac{n\pi}{2a\theta'(T)}+O(\log_2 T)\\
&\leqslant&n\log 2 -\frac{2\pi c K}{\log_3 T}\left(1+O\left(\frac{1}{\log_3T}\right)\right).\\
\end{eqnarray*}
We complete the proof by noting that $n =O(K(\log_2 T)^{-1}\log_3 T)$.
\end{proof}
Observe that the proof of the lower bound (\ref{Boundbeta1bis}) uses only the mean of $\sin \left(\pi \frac{x_k}{2a} \right)$. We now compute a lower bound for $\beta_1^+$ which takes into account the distribution of zeros.\\
Under the assumptions of Lemma \ref{DeltaS}, there exist, for $c\in \, ]0,1[$, arbitrary large $T$ and $a\asymp (\log_2 T)^{-1}\log_3 T$ such that $S(T+a)=cM_+(T+a)$ and $S(T-a)\ll(\log_3 T)^{-1}M_+(T+a)$. We select such a $T$ and $a$, and as suggested by the computations of \cite{Bober}, we assume, which is not essential, that
\begin{equation}\label{condS}
0\leqslant S(T-a)\leqslant S(\gamma_k)\hspace{3 mm}\mbox{for}\hspace{3 mm}k=0,\ldots,n.
\end{equation}
For ease of notation, we set 
$
 \tau_k=\sin \left(\pi \frac{\gamma_{k}-T}{2a} \right)\,\,\mbox{for}\,\,k=0,\ldots,n.
$
According to the position of $\gamma_{-1}$ with respect to $T-a$ and thanks to (\ref{eq:N}), we have
\[
k+1=\frac{1}{\pi}(\theta(\gamma_{k})-\theta(T-a))+S(\gamma_{k})-S(T-a)-q
\]
where $q=0$ or $1$, for $k=0,\ldots,n$. Therefore
\[
k+1=\frac{1}{\pi}\theta'(T)(\gamma_{k}-(T-a))+S(\gamma_{k})-S(T-a)-q-r
\]
where $0\leqslant r\leqslant\frac{a^2}{T}$ since $\theta'''(t)<0$ for $t\geqslant 1$,
and this implies that
\[
\gamma_{k}-T=-a+\frac{\pi}{\theta'(T)}\left(k+1+S(T-a)-S(\gamma_{k})+q+r\right).
\]
Thanks to (\ref{condS}) we have
\[
-M_+(T+a)\leqslant S(T-a)-S(\gamma_{k})\leqslant 0
\] and since $-1<\tau_k<1$ we get
\begin{equation}\label{eq:tau}
\tau_-(k)\leqslant \tau_k\leqslant \tau_+(k)
\end{equation}
where
\[
\tau_-(k)=-\cos\left(\max\left(\frac{\pi^2}{2a\theta'(T)}\left(k+1-M_+(T+a)\right),0\right)\right)
\]
and
\[
\tau_+(k)=-\cos\left(\min\left(\frac{\pi^2}{2a\theta'(T)}\left(k+2+r\right),\pi\right)\right)
\]
for $k=0,\ldots,n$.
For further use we recall some elementary facts concerning the divided differences.
\begin{lem}\label{2}
Let $I=(-1,1)$, $f\in C^{n+2}(I)$ and let $g$ be the function defined for pairwise distinct numbers $y_{0},\ldots,y_{n}\in I$ by
\[
g(y_{0},\ldots,y_{n})=\sum_{k=0}^{n}\frac{f(y_k)}{\displaystyle \prod_{\genfrac{}{}{0 cm}{} {0\leq j\leq n}{\; j\neq k}}(y_k-y_j)}
\]
and let $g^*$ be the continuous extension of $g$ defined for $y_{0},\ldots, y_{n}\in I$.
Then 
\begin{enumerate}
\item[a)] There exists $\eta=\eta(y_{0},\ldots,y_{n})\in I$ such that
\begin{equation}\label{eq:gp} 
\frac{\partial g^*}{\partial y_i}(y_{0},\ldots,y_{n})=\frac{f^{(n+1)}(\eta)}{(n+1)!}.
\end{equation}
\item[b)]
Let $h$ be the function defined in a neighbourhood of $0$ by 
\[
h(t)=g^*(y_0,\ldots,y_i+t,\ldots,y_j-t,\ldots,y_{n}).
\]
Then there exists $\xi=\xi(y_{0},\ldots,y_{n})\in I$ such that 
\begin{equation}\label{eq:h}
h'(0)=-(y_j-y_i)\frac{f^{(n+2)}(\xi)}{(n+2)!}.
\end{equation}
\end{enumerate}
\end{lem}
\begin{proof}
Assertion a) is a consequence of the representation formula 
\[
g(y_{0},\ldots,y_{n})=
\]
\[
\int\limits_0^1\,d\tau_1\,\int\limits_0^{\tau_1}\,d\tau_2\cdots\int\limits_0^{\tau_{n-1}}\,
f^{(n)}\,(y_{0}+\sum_{k=1}^{n}\,\tau_k(y_{k}-y_{k-1}))\, d\tau_{n}\,.
\]
Since divided differences are invariant by permutation it is sufficient to prove the second assertion for $i=n-1$ and $j=n$. By using integration by parts we get
\[
h'(0)=\int\limits_0^1\,d\tau_1\,\int\limits_0^{\tau_1}\,d\tau_2\cdots\int\limits_0^{\tau_{n-1}}\,
f^{(n+1)}\,(y_{0}+\sum_{k=1}^{n}\,\tau_k(y_{k}-y_{k-1}))(\tau_{n-1}-2\tau_{n})\, d\tau_{n}
\]
\[
=-(y_n-y_{n-1})\int\limits_0^1\,d\tau_1\,\int\limits_0^{\tau_1}\,d\tau_2\cdots\int\limits_0^{\tau_{n-1}}\,
f^{(n+2)}\,(y_{0}+\sum_{k=1}^{n}\,\tau_k(y_{k}-y_{k-1}))(\tau_{n-1}\tau_n-\tau_{n}^2)\, d\tau_{n}
\]
and (\ref{eq:h})  is a consequence of the mean value theorem. 
\end{proof}
 Let $I=(-1,1)$ and let $g$ be the function defined for pairwise distinct numbers $y_{0},\ldots,y_{n}\in I$ by
 \begin{equation}\label{diffdiv}
 g(y_{0},\ldots,y_{n})=\sum_{k=0}^{n}\frac{f(y_k)}{\displaystyle \prod_{\genfrac{}{}{0 cm}{} {0\leq j\leq n}{\; j\neq k}}(y_k-y_j)}
 \end{equation}
 where $f(t)= B_{2}\left(\frac{1}{2}+\frac{1}{\pi} \textnormal{Arcsin}\,\sqrt{\frac{1 +t}{2}}\,\right)$ and 
let $g^*$ be the continuous extension of $g$ defined for $y_{0},\ldots, y_{n}\in I$. Consider the problem (P)\,:
\[
\min g^*(y_0,\ldots,y_{n})
\]
subject to
\[
\left\lbrace
\begin{array}{c}
\vspace{3mm}
\tau_0\leqslant y_0\leqslant\ldots\leqslant y_n\leqslant \tau_{n},\\ \vspace{2mm}
\tau_-(k)\leqslant y_k\leqslant \tau_+(k)\,\,\mbox{for}\,\,k=0,\ldots,n,\\
 \displaystyle \sum_{k=0}^{n}y_k=\sum_{k=0}^{n}\tau_k\,.\\
\end{array}
\right.
\]
\begin{lem}\label{opt}
Problem (P) has a unique solution $(\underline{y}_0,\ldots,\underline{y}_{n})$. Further there exist $1\leqslant J\leqslant L\leqslant n-1$ such that $\underline{y}_k=\tau_+(k)$ for $k=0,\ldots,J-1$, $\underline{y}_J=\ldots =\underline{y}_L$ and $\underline{y}_k=\tau_-(k)$ for $k=L+1,\ldots,n$.
\end{lem}
\begin{proof}
Since $-1<\tau_0<\tau_{n}<1$ the function $g^*$ is continuous on the set $[\tau_0,\tau_{n}]^{n+1}$ and Problem (P) has at least a solution $(\underline{y}_0,\ldots,\underline{y}_{n})$. Moreover there exists $1\leqslant K\leqslant n-1$ such that $\tau_-(K)<\underline{y}_K<\tau_+(K)$. Let $J$ and $L$ be respectively the smallest and largest index such that $\underline{y}_k=\underline{y}_K$. Assume there exist indices $k\geqslant L+1$ such that $\tau_-(k)<\underline{y}_k$ and let $j$ be the smallest of these indices. Hence, for $t$ sufficiently small $(\underline{y}_0,\ldots,\underline{y}_L+t,\ldots,\underline{y}_j-t,\ldots,\underline{y}_{n})$ is an  admissible solution and thanks to (\ref{eq:h}) and \cite[Lemma 2.4]{Blanc2}  the function $h$ defined by
\[
h(t)=g^*(\underline{y}_0,\ldots,\underline{y}_L+t,\ldots,\underline{y}_j-t,\ldots,\underline{y}_{n})
\]
satisfies
\[
h'(0)=-(\underline{y}_j-\underline{y}_L)\frac{f^{(n+2)}(\xi)}{(n+2)!}<0.
\]
This is a contradiction and therefore  $\underline{y}_k=\tau_-(k)$ for $k=L+1,\ldots,n$. A very similar argument shows that $\underline{y}_k=\tau_+(k)$ for $k=0,\ldots,J-1$. The uniqueness is a consequence of the characterisation of the optimal solutions.
\end{proof}
 We are now able to compute a lower bound for $\beta_{1}^+$. By continuity, we can assume that the zeros are distinct. We have 
 \begin{eqnarray*}
 \beta_{1}^+&=& \Psi_{1}^*(x_{0},\ldots,x_n,a)\theta'(T)\\
 &=&4a\left(\sum_{k=0}^{n}\,\mu_k B_{2}\left(\frac{3}{4}+\frac{x_k}{4a}\right)\right)\theta'(T)
  \end{eqnarray*}
where we have used the fact that $B_{2}(\frac{1}{2}+x)$ is even, and the identity
  \[
   \frac{3}{4}+\frac{1}{2\pi}\mbox{Arcsin}\,t=\frac{1}{2}+
   \frac{1}{\pi}\mbox{Arcsin}\,\sqrt{\frac{1+ t}{2}}\hspace{2 mm} \mbox{for} \hspace{2 mm} t\in[-1,1]
   \]
 leads to
 \[
  \beta_1^+=4a\left(\sum_{k=0}^{n}\frac{f(\tau_k)}{\displaystyle \prod_{\genfrac{}{}{0 cm}{} {0\leq j\leq n}{\; j\neq k}}(\tau_k-\tau_j)}\right)\theta'(T)\geqslant 4a\,g^*(\underline{y}_0,\ldots,\underline{y}_{n})\theta'(T)
 \]
where $(\underline{y}_0,\ldots,\underline{y}_{n})$ is the solution of Problem (P). Note that this solution depends on $c$, $T$, $a$, $M_+(T+a)$,  $\sum_{k=0}^{n}\tau_k $ and $n$. Thanks to (\ref{eq:N}) and Lemma \ref{meanxk} the lower bound takes the form
\begin{equation}\label{fin}
\beta_1^+\geqslant \phi_c(T,a,M_+(T+a),r_1(T,a),r_2(T,a))
\end{equation}
where $r_1(T,a)$ and $r_2(T,a)$ are $O((\log_3 T)^{-1})$ and to get an approximation for very large $T$, it makes sense to put $r_1=r_2=0$ in (\ref{fin}).
To compute numerically a lower bound for $\phi_c(T,a,M_+(T+a),0,0)$, we first solve Problem (P) and thanks to (\ref{eq:gp}) and \cite[Lemma 2.4]{Blanc2} we can slightly decrease the $\underline{y}_k$ to make them distinct and we use relation (\ref{diffdiv}). To conclude we present an argument which shows why $\vert d_{2k-1}^+\vert$ is very probably small with respect to $d_1^+$ if $d_{2k-1}^+$ is negative. Let 
\[
s_n(t)=-\frac{2}{\sqrt{n}}(\theta'(t)-\log n)\sin(\theta(t)-t\log n)
\]
so that 
\[
Z'(t)=\sum_{1\leq n \leq \sqrt{\frac{t}{2\pi}}}s_n(t)+O\left(t^{-\frac{1}{4}}\theta'(t)^2\right).
\]
For $T$ large, $0<a\ll 1$ and $K\!\in\!\!\left.\left(\frac{1}{4}\theta'(T),\frac{7}{8}\theta'(T)\right]\right.$ we have
\[
d_{2k-1}^+=\frac{1}{\theta'(T)^{2k-1}}\sum_{1\leq n \leq \sqrt{\frac{T+a}{2\pi}}}(\theta'(T+a)-\log n)^{(2k-2)}s_n(T+a)+O\left(T^{-\frac{1}{50}}{\theta'(T)}\right)
\]
and by partial summation
\[
d_{2k-1}^+\geqslant\left(1+O\left(\frac{a}{T}\right)\right)\frac{1}{\theta'(T)}\left(\min_{1\leqslant m\leqslant \sqrt{\frac{T+a}{2\pi}}}\;\;\sum_{n=1}^ms_n(T+a)\right)+O\left(T^{-\frac{1}{50}}{\theta'(T)}\right).
\]
Since $Z'(T+a)$ is very large, it is difficult to imagine that there exists $m$ such that  $\sum_{n=1}^ms_n(T+a)$ is negative and large in absolute value.

\end{document}